\documentclass[a4paper,12pt]{article}
\usepackage{amsmath}
\usepackage{amsfonts}

\newtheorem{theorem}{Theorem}

\newtheorem{corollary}[theorem]{Corollary}

\newtheorem{lemma}[theorem]{Lemma}

\newtheorem{proposition}[theorem]{Proposition}

\newenvironment{proof}[1][Proof]{\noindent \textbf{#1.} }{\  \rule{0.5em}{0.5em}}
\include {mak}
\input{tcilatex}
\begin{document}

\begin{center}
{\Large Super-congruences involving trininomial coefficients}

\  \ 

\textbf{Laid Elkhiri}

EDTNLHM Laboratory

Department of Mathematics, ENS Old Kouba, Algiers, Algeria

elkhirialjebre@hotmail.fr

and

\textbf{Miloud Mihoubi}

RECITS Laboratory

Faculty of Mathematics, USTHB, Algiers, Algeria

mmihoubi@usthb.dz \ or \ miloudmihoubi@gmail.com

\ 
\end{center}

\noindent \textbf{Abstract.} The aim of this work is to establish
congruences $\left( \func{mod}p^{2}\right) $ involving the trinomial
coefficients $\binom{np-1}{p-1}_{2}$ and $\binom{np-1}{\left( p-1\right) /2}%
_{2}$ arising from the expansion of the powers of the polynomial $1+x+x^{2}.$
In main results we extend some known congruences involving the binomial
coefficients $\binom{np-1}{p-1}$ and $\binom{np-1}{\left( p-1\right) /2}$
and establish congruences link binomial coefficients and harmonic numbers.

\noindent \textbf{Keywords.} Binomial coefficients, trinomial coefficients,
harmonic numbers, congruences.

\noindent \textbf{MSC: }11B65, 11A07, 05A10.

\section{Introduction and main results}

Great mathematicians studied in the 19-st century congruences of the forms $%
\binom{2p-1}{p-1}$ and $\binom{p-1}{\left( p-1\right) /2},$ in 1819, Babbage 
\cite{BAB} showed, for any prime number $p\geq 3,$ the congruence%
\begin{equation*}
\binom{2p-1}{p-1}\equiv 1\text{ }\left( \func{mod}p^{2}\right) .
\end{equation*}%
In 1862, Wolstenholme \cite{WOL} proved, for any prime number $p\geq 5,$
that the above congruence can be extended to%
\begin{equation*}
\binom{2p-1}{p-1}\equiv 1\text{ }\left( \func{mod}p^{3}\right)
\end{equation*}%
and in 1900, Glaisher \cite{GLA 1} proved, for any prime number $p\geq 5,$
that the above congruence can also be extended to%
\begin{equation*}
\binom{np-1}{p-1}\equiv 1\text{ }\left( \func{mod}p^{3}\right) ,\text{ }%
n\geq 1.
\end{equation*}%
In 1895, Morley \cite{MORL} proved, for any prime number $p\geq 5,$ that%
\begin{equation*}
\binom{p-1}{\left( p-1\right) /2}\equiv \left( -1\right) ^{\left( p-1\right)
/2}4^{p-1}=\left( -1\right) ^{\left( p-1\right) /2}\left( 1+pq_{2}\right)
^{2}\text{ }\left( \func{mod}p^{3}\right) ,
\end{equation*}%
where $q_{a}$ is the Fermat quotient defined for a given prime number $p$ by 
\begin{equation*}
q_{a}=q_{a}\left( p\right) :=\frac{a^{p-1}-1}{p},\  \  \ a\in \mathbb{Z-}p%
\mathbb{Z},
\end{equation*}%
and $\mathbb{Z}$ denotes the set of the integer numbers.\newline
Also, in 1953, Carlitz \cite{CAR1,CAR2} improved, for any prime number $%
p\geq 5,$ Morley's congruence to%
\begin{equation*}
\left( -1\right) ^{\frac{p-1}{2}}\binom{p-1}{\left( p-1\right) /2}\equiv
4^{p-1}+\frac{p^{3}}{12}\text{ }\left( \func{mod}p^{4}\right) .
\end{equation*}%
Many great mathematicians have been interested to generalize the congruence
of Wostenhlom and Morly, such the works of Zhao \cite{ZHA}, McIntosh \cite%
{MCI}, Me\v{s}trovi\'{c} \cite{ROM 2}, Bencherif et al. \cite{BEN} and Sun 
\cite{SUN}. Recently, Sun \cite{SUN1} gave some properties and congruences
involving the coefficients $\dbinom{n}{n}_{2}$ defined by%
\begin{equation*}
\left( 1+x+x^{2}\right) ^{n}=\overset{2n}{\underset{k=0}{\sum }}\binom{n}{k}%
_{2}x^{k}.
\end{equation*}%
see also Cao \& Pan \cite{CAO1} and Cao \& Sun \cite{CAO}.\newline
The idea of this work is inspired from the congruences given by Wolstenholme
and Morly. We study congruences modulo $p^{2}$ for the trinomial
coefficients $\binom{np-1}{p-1}_{2}$ and $\binom{np-1}{\left( p-1\right) /2}%
_{2}.$ We prove congruences involving trinomial coefficients, binomial
coefficients and harmonic numbers. Our main results are given as follows.

\begin{theorem}
\label{TH}Let $p\geq 5$ be a prime number and $n$ be a positive integer. We
have 
\begin{equation}
\binom{np-1}{p-1}_{2}\equiv \left \{ 
\begin{array}{c}
\text{ \ }1+npq_{3}\text{ }\left( \func{mod}p^{2}\right) \text{\  \  \ if \ }%
p\equiv 1\text{ }\left( \func{mod}3\right) ,\medskip \\ 
-1-npq_{3}\text{ }\left( \func{mod}p^{2}\right) \text{\  \  \ if \ }p\equiv 2%
\text{ }\left( \func{mod}3\right) .%
\end{array}%
\right.  \label{2}
\end{equation}%
and%
\begin{equation}
\binom{np-1}{\frac{p-1}{2}}_{2}\equiv \left \{ 
\begin{array}{c}
1+np\left( 2q_{2}+\frac{1}{2}q_{3}\right) \text{ }\left( \func{mod}%
p^{2}\right) \text{\  \  \ if \ }p\equiv 1\text{ }\left( \func{mod}6\right)
,\medskip \\ 
\text{ \  \  \  \ }-\frac{1}{2}npq_{3}\text{ }\left( \func{mod}p^{2}\right) 
\text{ \  \  \  \  \  \  \  \  \ if \ }p\equiv 5\text{ }\left( \func{mod}6\right) .%
\end{array}%
\right.  \label{4}
\end{equation}
\end{theorem}

\begin{theorem}
\label{C} For every prime number $p\geq 5$ we have%
\begin{equation}
\dsum \limits_{k=0}^{\frac{p-1}{2}}\binom{2k}{k}H_{k}\equiv \left \{ 
\begin{array}{l}
-q_{3}\text{ }\left( \func{mod}p\right) \text{ if }p\equiv 1\text{ }\left( 
\func{mod}3\right) ,\medskip \\ 
\text{ \ }q_{3}\text{ }\left( \func{mod}p\right) \text{ if }p\equiv 2\text{ }%
\left( \func{mod}3\right) ,%
\end{array}%
\right.  \label{6}
\end{equation}%
and%
\begin{equation}
\dsum \limits_{k=1}^{\left[ \frac{p-1}{4}\right] }\frac{1}{4^{k}}\dbinom{4k}{%
2k}\left( 2H_{2k}-H_{k}\right) \equiv \left \{ 
\begin{array}{l}
-\left( -1\right) ^{\left( p-1\right) /2}\frac{q_{3}}{2}\text{ }\left( \func{%
mod}p\right) \  \text{if }p\equiv 1\text{ }\left( \func{mod}6\right) ,\medskip
\\ 
\text{ \ }\left( -1\right) ^{\left( p-1\right) /2}\frac{q_{3}}{2}\text{ }%
\left( \func{mod}p\right) \  \text{if }p\equiv 5\text{ }\left( \func{mod}%
6\right) ,%
\end{array}%
\right.  \label{7}
\end{equation}%
where $H_{n}$ to be the $n$-th harmonic number defined by%
\begin{equation*}
H_{0}=0,\  \  \ H_{n}=1+\frac{1}{2}+\cdots +\frac{1}{n}.
\end{equation*}
\end{theorem}

\begin{proposition}
\label{PP}Let $p\geq 5$ be a prime number and $n$ be a positive integer.
Then 
\begin{equation}
\overset{p-1}{\underset{k=0}{\sum }}\binom{np-1}{k}_{2}\equiv \left \{ 
\begin{array}{l}
1+npq_{3}\text{ }\left( \func{mod}p^{2}\right) \text{ \  \  \ if }p\equiv 1%
\text{ }\left( \func{mod}3\right) ,\medskip \\ 
\text{ \  \  \  \ }0\text{ \  \  \  \ }\left( \func{mod}p^{2}\right) \text{ \  \  \
if }p\equiv 2\text{ }\left( \func{mod}3\right)%
\end{array}%
\right.  \label{9}
\end{equation}%
and%
\begin{equation}
\overset{\frac{p-1}{2}}{\underset{k=0}{\sum }}\binom{np-1}{k}_{2}\equiv
\left \{ 
\begin{array}{l}
1+np\left( \frac{4}{3}q_{2}+q_{3}\right) \text{ }\left( \func{mod}%
p^{2}\right) \text{ \  \  \ if }p\equiv 1\text{ }\left( \func{mod}6\right)
,\medskip \\ 
\text{ \  \  \ }-\frac{2}{3}npq_{2}\text{ \  \  \  \  \  \  \ }\left( \func{mod}%
p^{2}\right) \text{ \  \  \ if }p\equiv 5\text{ }\left( \func{mod}6\right) .%
\end{array}%
\right.  \label{10}
\end{equation}
\end{proposition}

\noindent For $k\leq p-1,$ since%
\begin{equation}
\binom{np-1}{k}=\left( -1\right) ^{k}\overset{k}{\underset{i=1}{\prod }}%
\left( 1-\frac{np}{i}\right) \equiv \left( -1\right) ^{k}\left(
1-npH_{k}\right) \text{ }\left( \func{mod}p^{2}\right)  \label{C55}
\end{equation}%
we conclude that $\binom{np^{2}-1}{k}\equiv \left( -1\right) ^{k}$ $\left( 
\func{mod}p^{2}\right) .$\newline
A similar congruence for the coefficients $\binom{np^{2}-1}{k}_{2}$ is given
as follows:

\begin{corollary}
\label{CC}Let $p\geq 5$ be a prime number and $n,k$ be integers with $n\geq
1 $ and $k\in \left \{ 0,1,\ldots ,p-1\right \} .$ We have 
\begin{equation}
\binom{np^{2}-1}{k}_{2}\equiv \left \{ 
\begin{array}{l}
\text{ \ }1\text{ }\left( \func{mod}p^{2}\right) \text{ \  \ if \ }k\equiv 0%
\text{ }\left( \func{mod}3\right) ,\medskip \\ 
-1\text{ }\left( \func{mod}p^{2}\right) \text{ \  \ if \ }k\equiv 1\text{ }%
\left( \func{mod}3\right) ,\medskip \\ 
\text{ \ }0\text{ }\left( \func{mod}p^{2}\right) \text{ \  \ if \ }k\equiv 2%
\text{ }\left( \func{mod}3\right) .%
\end{array}%
\right.  \label{11}
\end{equation}
\end{corollary}

\section{Some basic congruences}

In this section, we give some congruences involving harmonic numbers and
trinomial coefficients in order to prove the main theorems.

\begin{lemma}
\cite{EIN,GL,LER} Let $p$ be a prime number. We have%
\begin{eqnarray}
H_{\left[ p/2\right] } &\equiv &-2q_{2}\text{ }\left( \func{mod}p\right) ,\
\  \ p\geq 3,  \label{GL0} \\
H_{\left[ p/3\right] } &\equiv &-\frac{3}{2}q_{3}\text{ }\left( \func{mod}%
p\right) ,\  \  \ p\geq 5,  \label{GL} \\
H_{\left[ p/6\right] } &\equiv &-2q_{2}-\frac{3}{2}q_{3}\text{ }\left( \func{%
mod}p\right) ,\  \  \ p\geq 5.  \label{GL2}
\end{eqnarray}
\end{lemma}

\begin{lemma}
For any prime number $p\geq 3$ we have%
\begin{eqnarray}
H_{p-k} &\equiv &H_{k-1}\text{ }\left( \func{mod}p\right) ,\text{ \  \ }1\leq
k\leq p-1,  \label{CONG 0} \\
H_{\frac{p-1}{2}-k} &\equiv &-2q_{2}+2H_{2k}-H_{k}\text{ }\left( \func{mod}%
p\right) ,\text{ \ }1\leq k\leq \frac{p-1}{2}.  \label{CONG 1}
\end{eqnarray}
\end{lemma}

\begin{proof}
When $k\in \left \{ 1,2,\ldots ,p-1\right \} ,$ it is obvious that we have%
\begin{equation*}
H_{p-k}=\sum \limits_{i=1}^{p-k}\frac{1}{i}=\sum \limits_{i=1}^{p-1}\frac{1}{%
i}-\sum \limits_{i=p-k+1}^{p-1}\frac{1}{i}=H_{p-1}-\sum \limits_{i=1}^{k-1}%
\frac{1}{p-k+i}.
\end{equation*}%
Then, since $H_{p-1}\equiv 0$ $\left( \func{mod}p\right) $ we get $%
H_{p-k}\equiv \sum \limits_{i=1}^{k-1}\frac{1}{k-i}=H_{k-1}$ $\left( \func{%
mod}p\right) .$

\noindent Similarly, if $k\in \left \{ 1,2,\ldots ,\left( p-1\right)
/2\right \} ,$ we get%
\begin{eqnarray*}
H_{\left( p-1\right) /2-k} &=&\dsum \limits_{j=1}^{\left( p-1\right) /2-k}%
\frac{1}{j}=\dsum \limits_{j=1}^{\left( p-1\right) /2}\frac{1}{j}-\dsum
\limits_{j=\left( p-1\right) /2-k+1}^{\left( p-1\right) /2}\frac{1}{j} \\
&=&H_{\left( p-1\right) /2}-\frac{2}{p-1}-\dsum \limits_{j=\left( p-1\right)
/2-k+1}^{\left( p-3\right) /2}\frac{1}{j} \\
&=&H_{\left( p-1\right) /2}-\frac{2}{p-1}-\dsum \limits_{j=1}^{k-1}\frac{1}{%
\left( p-1\right) /2-j} \\
&=&H_{\left( p-1\right) /2}-\frac{2}{p-1}-\dsum \limits_{j=1}^{k-1}\frac{2}{%
p-1-2j},
\end{eqnarray*}%
and since $H_{\left( p-1\right) /2}\equiv -2q_{2}$ $\left( \func{mod}%
p\right) $ \cite{EIN}, we conclude that%
\begin{eqnarray*}
H_{\left( p-1\right) /2-k} &\equiv &-2q_{2}+2+2\dsum \limits_{j=1}^{k-1}%
\frac{1}{2j+1} \\
&=&-2q_{2}+2+2\left( \dsum \limits_{j=1}^{2k-1}\frac{1}{j}-\dsum
\limits_{j=1}^{k-1}\frac{1}{2j}-1\right) \\
&=&-2q_{2}+2H_{2k-1}-H_{k-1} \\
&=&-2q_{2}+2H_{2k}-H_{k}\text{ }\left( \func{mod}p\right) .
\end{eqnarray*}
\end{proof}

\begin{lemma}
Let $p\geq 5$ be a prime number. Then, if $p\equiv 1$ $\left( \func{mod}%
3\right) $ we obtain%
\begin{eqnarray}
\dsum \limits_{k=0}^{\left( p-4\right) /3}\frac{1}{3k+2} &\equiv &0\text{ }%
\left( \func{mod}p\right) ,  \label{C1b} \\
\dsum \limits_{k=0}^{\left( p-4\right) /3}\frac{1}{3k+1} &\equiv &\frac{1}{2}%
q_{3}\text{ }\left( \func{mod}p\right) ,  \label{C1c}
\end{eqnarray}%
and if $p\equiv 2$ $\left( \func{mod}3\right) $ we obtain%
\begin{eqnarray}
\dsum \limits_{k=0}^{\left( p-5\right) /3}\frac{1}{3k+1} &\equiv &1\text{ }%
\left( \func{mod}p\right) ,  \label{C2b} \\
\dsum \limits_{k=0}^{\left( p-5\right) /3}\frac{1}{3k+2} &\equiv &\frac{1}{2}%
q_{3}\text{ }\left( \func{mod}p\right) .  \label{C2c}
\end{eqnarray}
\end{lemma}

\begin{proof}
For any prime number $p\equiv 1$ $\left( \func{mod}3\right) ,$ we have%
\begin{equation*}
\dsum \limits_{k=0}^{\left( p-4\right) /3}\frac{1}{3k+2}=\dsum
\limits_{k=0}^{\left( p-4\right) /3}\frac{1}{3\left( \frac{p-4}{3}-k\right)
+2}\equiv -\dsum \limits_{k=0}^{\left( p-4\right) /3}\frac{1}{2+3k}\text{ }%
\left( \func{mod}p\right)
\end{equation*}%
and this gives the congruence (\ref{C1b}). From the identity%
\begin{equation*}
\dsum \limits_{k=0}^{\left( p-4\right) /3}\frac{1}{3k+1}+\dsum
\limits_{k=0}^{\left( p-4\right) /3}\frac{1}{3k+2}+\dsum
\limits_{k=0}^{\left( p-4\right) /3}\frac{1}{3k+3}=\dsum \limits_{k=1}^{p-1}%
\frac{1}{k}
\end{equation*}%
and by the congruences $H_{p-1}\equiv 0$ $\left( \func{mod}p\right) $ and (%
\ref{GL}) it results 
\begin{equation*}
\dsum \limits_{k=0}^{\left( p-4\right) /3}\frac{1}{3k+1}=H_{p-1}-\frac{1}{3}%
H_{\left[ p/3\right] }-\dsum \limits_{k=0}^{\left( p-4\right) /3}\frac{1}{%
3k+2}\equiv 0+\frac{1}{2}q_{3}-0\equiv \frac{1}{2}q_{3}\text{ }\left( \func{%
mod}p\right)
\end{equation*}%
which gives the congruence (\ref{C1c}). \newline
Also, if $p\equiv 2$ $\left( \func{mod}3\right) ,$ the other congruences can
be proved similarly.
\end{proof}

\begin{lemma}
For any prime number $p\geq 5$ we have%
\begin{eqnarray}
\dsum \limits_{k=0}^{\left( p-1\right) /6}\frac{1}{2k+1} &\equiv &q_{2}-%
\frac{3}{4}q_{3}+\frac{3}{2}\text{ }\left( \func{mod}p\right) \text{ \ if }%
p\equiv 1\text{ }\left( \func{mod}6\right) ,  \label{C3} \\
\dsum \limits_{k=0}^{\left( p-5\right) /6}\frac{1}{2k+1} &\equiv &q_{2}-%
\frac{3}{4}q_{3}\text{ }\left( \func{mod}p\right) \text{ \ if }p\equiv 5%
\text{ }\left( \func{mod}6\right) .  \label{C3b}
\end{eqnarray}
\end{lemma}

\begin{proof}
For $p\equiv 1$ $\left( \func{mod}6\right) $ use the congruence (\ref{GL})
to obtain%
\begin{equation*}
\dsum \limits_{k=0}^{\left( p-1\right) /6}\frac{1}{2k+1}+\dsum
\limits_{k=1}^{\left( p-1\right) /6}\frac{1}{2k}=\dsum \limits_{k=1}^{\left(
p-1\right) /3+1}\frac{1}{k}=H_{\left[ p/3\right] }+\frac{3}{p+2}\equiv -%
\frac{3}{2}q_{3}+\frac{3}{2}\text{ }\left( \func{mod}p\right) ,
\end{equation*}%
and for $p\equiv 5$ $\left( \func{mod}6\right) $ use the congruence (\ref{GL}%
) to obtain%
\begin{equation*}
\dsum \limits_{k=0}^{\left( p-5\right) /6}\frac{1}{2k+1}+\dsum
\limits_{k=1}^{\left( p-5\right) /6}\frac{1}{2k}=\dsum \limits_{k=1}^{\left(
p-2\right) /3}\frac{1}{k}=H_{\left[ p/3\right] }\equiv -\frac{3}{2}q_{3}%
\text{ }\left( \func{mod}p\right) .
\end{equation*}%
By using in the last congruences the congruence (\ref{GL2}) 
\begin{equation*}
\dsum \limits_{k=0}^{\left[ p/6\right] }\frac{1}{2k}=\frac{1}{2}H_{\left[ p/6%
\right] }\equiv -q_{2}-\frac{3}{4}q_{3}\text{ }\left( \func{mod}p\right) ,
\end{equation*}%
the desired congruences follow.
\end{proof}

\begin{lemma}
Let $p$ be a prime number. Then, for $p\equiv 1$ $\left( \func{mod}6\right) $
we have%
\begin{eqnarray}
\dsum \limits_{k=0}^{\left( p-1\right) /6}\frac{1}{3k+1} &\equiv &-\frac{2}{3%
}q_{2}+2\text{ }\left( \func{mod}p\right) ,  \label{H0} \\
\dsum \limits_{k=0}^{\left( p-1\right) /6}\frac{1}{3k+2} &\equiv &-\frac{2}{3%
}q_{2}+\frac{1}{2}q_{3}+\frac{2}{3}\text{ }\left( \func{mod}p\right) ,
\label{H1}
\end{eqnarray}%
and, for $p\equiv 5$ $\left( \func{mod}6\right) $ we have 
\begin{eqnarray}
\dsum \limits_{k=0}^{\left( p-5\right) /6}\frac{1}{3k+1} &\equiv &\frac{1}{2}%
q_{3}-\frac{2}{3}q_{2}\text{ }\left( \func{mod}p\right) ,  \label{H3} \\
\dsum \limits_{k=0}^{\left( p-5\right) /6}\frac{1}{3k+2} &\equiv &-\frac{2}{3%
}q_{2}\text{ }\left( \func{mod}p\right) .  \label{H2}
\end{eqnarray}
\end{lemma}

\begin{proof}
For $p\equiv 1$ $\left( \func{mod}6\right) ,$ by the congruences (\ref{C2c})
and (\ref{C3}) we get%
\begin{eqnarray*}
\dsum \limits_{k=0}^{\left( p-1\right) /6}\frac{1}{3k+2} &=&2\dsum%
\limits_{k=0}^{\left( p-1\right) /6}\frac{1}{6k+4} \\
&=&2\dsum \limits_{k=0}^{\left( p-1\right) /3}\frac{1}{6k+4}-2\dsum
\limits_{k=\left( p-1\right) /6+1}^{\left( p-1\right) /3}\frac{1}{6k+4} \\
&=&\dsum \limits_{k=0}^{\left( p-1\right) /3}\frac{1}{3k+2}-2\dsum
\limits_{k=1}^{\left( p-1\right) /6}\frac{1}{6k+p+3} \\
&\equiv &1-\frac{2}{3}\dsum \limits_{k=1}^{\left( p-1\right) /6}\frac{1}{2k+1%
} \\
&\equiv &1-\frac{2}{3}\left( q_{2}-\frac{3}{4}q_{3}+\frac{1}{2}\right) \\
&=&-\frac{2}{3}q_{2}+\frac{1}{2}q_{3}+\frac{2}{3}\text{ }\left( \func{mod}%
p\right) ,
\end{eqnarray*}%
We also have%
\begin{equation*}
\dsum \limits_{k=1}^{\left( p+5\right) /2}\frac{1}{k}=\dsum
\limits_{k=0}^{\left( p-1\right) /6}\frac{1}{3k+1}+\dsum
\limits_{k=0}^{\left( p-1\right) /6}\frac{1}{3k+2}+\frac{1}{3}\dsum
\limits_{k=0}^{\left( p-1\right) /6}\frac{1}{k+1}
\end{equation*}%
which gives on using the congruences (\ref{GL0}), (\ref{GL2}) and (\ref{H1})%
\begin{eqnarray*}
\dsum \limits_{k=0}^{\left( p-1\right) /6}\frac{1}{3k+1} &\equiv &\dsum
\limits_{k=1}^{\left( p+5\right) /2}\frac{1}{k}-\dsum \limits_{k=0}^{\left(
p-1\right) /6}\frac{1}{3k+2}-\frac{1}{3}\dsum \limits_{k=0}^{\left(
p-1\right) /6}\frac{1}{k+1} \\
&=&H_{\left[ p/2\right] }+\frac{2}{p+5}+\frac{2}{p+3}+\frac{2}{p+1}-\dsum
\limits_{k=0}^{\left( p-1\right) /6}\frac{1}{3k+2}-\frac{1}{3}\left( H_{%
\left[ p/6\right] }+\frac{6}{p+5}\right) \\
&\equiv &-2q_{2}+\frac{8}{3}-\left( -\frac{2}{3}q_{2}+\frac{1}{2}q_{3}+\frac{%
2}{3}\right) -\frac{1}{3}\left( -2q_{2}-\frac{3}{2}q_{3}\right) \\
&=&-\frac{2}{3}q_{2}+2\text{ }\left( \func{mod}p\right) .
\end{eqnarray*}%
For $p\equiv 5$ $\left( \func{mod}6\right) $ use the congruence (\ref{C2c})
to get%
\begin{equation*}
\dsum \limits_{k=0}^{\left( p-5\right) /6}\frac{1}{3k+2}=2\dsum%
\limits_{k=0}^{\left( p-5\right) /3}\frac{1}{6k+4}-2\dsum \limits_{k=\left(
p-5\right) /6+1}^{\left( p-5\right) /3}\frac{1}{6k+4}
\end{equation*}%
\begin{eqnarray*}
&=&\dsum \limits_{k=0}^{\left( p-5\right) /3}\frac{1}{3k+2}-2\dsum
\limits_{k=1}^{\left( p-5\right) /6}\frac{1}{6k+p-1} \\
&\equiv &\frac{1}{2}q_{3}-2\dsum \limits_{k=1}^{\left( p-5\right) /6}\frac{1%
}{6k-1}\text{ }\left( \func{mod}p\right) .
\end{eqnarray*}%
by setting $k=\left( p+1\right) /6-j$ and using (\ref{GL2}) this last
congruence becomes 
\begin{equation*}
\dsum \limits_{k=1}^{\left( p-5\right) /6}\frac{1}{6k-1}\equiv -\frac{1}{6}%
\dsum \limits_{j=1}^{\left( p-5\right) /6}\frac{1}{j}=-\frac{1}{6}H_{\left[
p/6\right] }\equiv \frac{1}{3}q_{2}+\frac{1}{4}q_{3}\text{ }\left( \func{mod}%
p\right) ,
\end{equation*}%
hence $\dsum \limits_{k=0}^{\left( p-5\right) /6}\frac{1}{3k+2}\equiv \frac{1%
}{2}q_{3}-2\left( \frac{1}{3}q_{2}+\frac{1}{4}q_{3}\right) \equiv -\frac{2}{3%
}q_{2}$ $\left( \func{mod}p\right) .$ We also have%
\begin{equation*}
\dsum \limits_{k=0}^{\left( p-5\right) /6}\frac{1}{3k+1}+\dsum
\limits_{k=0}^{\left( p-5\right) /6}\frac{1}{3k+2}+\frac{1}{3}\dsum
\limits_{k=0}^{\left( p-5\right) /6}\frac{1}{k+1}=\dsum
\limits_{k=1}^{\left( p+1\right) /2}\frac{1}{k}
\end{equation*}%
and by using the congruences (\ref{GL0}), (\ref{GL}) and (\ref{H2}) this
gives%
\begin{eqnarray*}
\dsum \limits_{k=0}^{\left( p-5\right) /6}\frac{1}{3k+1} &=&\left( \frac{2}{%
p+1}+H_{\left[ p/2\right] }\right) -\frac{1}{3}\left( \frac{6}{p+1}+H_{\left[
p/6\right] }\right) -\dsum \limits_{k=0}^{\left( p-5\right) /6}\frac{1}{3k+2}
\\
&\equiv &\frac{1}{2}q_{3}-\frac{2}{3}q_{2}\text{ }\left( \func{mod}p\right) .
\end{eqnarray*}
\end{proof}

\begin{proposition}
\label{P}Let $p\geq 5$ be a prime number and $n,k$ be positive integers. We
have%
\begin{eqnarray}
\binom{np-1}{3k}_{2} &\equiv &1-np\left( \frac{2}{3}H_{k}+\dsum%
\limits_{j=0}^{k-1}\frac{1}{3j+2}\right) \text{ }\left( \func{mod}%
p^{2}\right) ,\  \  \ 3k\leq p-1,  \label{CC1} \\
\binom{np-1}{3k+1}_{2} &\equiv &-1+np\left( \frac{2}{3}H_{k}+\dsum%
\limits_{j=0}^{k}\frac{1}{3j+1}\right) \text{ }\left( \func{mod}p^{2}\right)
,\  \  \ 3k+1\leq p-1,  \label{CC2} \\
\binom{np-1}{3k+2}_{2} &\equiv &np\left( -\dsum \limits_{j=0}^{k}\frac{1}{%
3j+1}+\dsum \limits_{j=0}^{k}\frac{1}{3j+2}\right) \text{ }\left( \func{mod}%
p^{2}\right) ,\  \  \ 3k+2\leq p-1.  \label{CC3}
\end{eqnarray}
\end{proposition}

\begin{proof}
From the expansion 
\begin{equation*}
\left( 1+x+x^{2}\right) ^{n}=\left( x+e^{i\frac{\pi }{3}}\right) ^{n}\left(
x+e^{-i\frac{\pi }{3}}\right) ^{n}=\dsum \limits_{k\geq 0}\left( \dsum
\limits_{j=0}^{k}\binom{n}{j}\binom{n}{k-j}e^{\left( k-2j\right) i\frac{\pi 
}{3}}\right) x^{k}
\end{equation*}%
we deduce the identity%
\begin{equation}
\binom{n}{k}_{2}=\dsum \limits_{j=0}^{k}\binom{n}{j}\binom{n}{k-j}\cos \frac{%
\left( k-2j\right) \pi }{3}.  \label{ID5}
\end{equation}%
and by the congruence (\ref{C55}) and the identity (\ref{ID5}) we get%
\begin{eqnarray*}
\binom{np-1}{k}_{2} &=&\dsum \limits_{j=0}^{k}\binom{np-1}{j}\binom{np-1}{k-j%
}\cos \frac{\left( k-2j\right) \pi }{3} \\
&\equiv &\left( -1\right) ^{k}\dsum \limits_{j=0}^{k}\left( 1-np\left(
H_{j}+H_{k-j}\right) \right) \cos \frac{\left( k-2j\right) \pi }{3} \\
&=&\left( -1\right) ^{k}\dsum \limits_{j=0}^{k}\left( 1-2npH_{j}\right) \cos 
\frac{\left( k-2j\right) \pi }{3}\text{ }\left( \func{mod}p^{2}\right)
\end{eqnarray*}%
Then, for the congruence (\ref{CC1}), we have%
\begin{eqnarray*}
\binom{np-1}{3k}_{2} &\equiv &\dsum \limits_{j=0}^{3k}\cos \frac{2j\pi }{3}%
-2np\dsum \limits_{j=0}^{3k}H_{j}\cos \frac{2j\pi }{3} \\
&=&1-2np\dsum \limits_{j=0}^{3k}H_{j}\cos \frac{2j\pi }{3} \\
&=&1-np\left( 2\dsum \limits_{j=0}^{k}H_{3j}-\dsum
\limits_{j=0}^{k-1}H_{3j+1}-\dsum \limits_{j=0}^{k-1}H_{3j+2}\right) \\
&=&1-np\left( \dsum \limits_{j=0}^{k-1}\left(
2H_{3j}-H_{3j+1}-H_{3j+2}\right) +2H_{3k}\right) \\
&=&1-2npH_{3k}+np\left( 2\dsum \limits_{j=0}^{k-1}\frac{1}{3j+1}+\dsum
\limits_{j=0}^{k-1}\frac{1}{3j+2}\right) \\
&=&1-2np\left( \dsum \limits_{j=0}^{k-1}\frac{1}{3j+1}+\dsum%
\limits_{j=0}^{k-1}\frac{1}{3j+2}+\dsum \limits_{j=1}^{k}\frac{1}{3j}\right)
\\
&&+np\left( 2\dsum \limits_{j=0}^{k-1}\frac{1}{3j+1}+\dsum
\limits_{j=0}^{k-1}\frac{1}{3j+2}\right) \\
&=&1-np\left( \frac{2}{3}H_{k}+\dsum \limits_{j=0}^{k-1}\frac{1}{3j+2}%
\right) \text{ }\left( \func{mod}p^{2}\right) .
\end{eqnarray*}%
For the congruence (\ref{CC2}) we have%
\begin{eqnarray*}
\binom{np-1}{3k+1}_{2} &\equiv &-\dsum \limits_{j=0}^{3k+1}\cos \frac{\left(
2j-1\right) \pi }{3}+2np\dsum \limits_{j=0}^{3k+1}H_{j}\cos \frac{\left(
2j-1\right) \pi }{3} \\
&=&-1+2np\dsum \limits_{j=0}^{3k+1}H_{j}\cos \frac{\left( 2j-1\right) \pi }{3%
} \\
&=&-1+np\left( \dsum \limits_{j=0}^{k}\left(
H_{3j}+H_{3j+1}-2H_{3j+2}\right) +2H_{3k+2}\right) \\
&=&-1+np\left( -\dsum \limits_{j=0}^{k}\frac{1}{3j+1}-2\dsum
\limits_{j=0}^{k}\frac{1}{3j+2}+2H_{3k}+\frac{2}{3k+1}+\frac{2}{3k+2}\right)
\end{eqnarray*}%
\begin{eqnarray*}
&=&-1+2npH_{3k}+np\left( -\dsum \limits_{j=0}^{k-1}\frac{1}{3j+1}-2\dsum
\limits_{j=0}^{k-1}\frac{1}{3j+2}+\frac{1}{3k+1}\right) \\
&=&-1+2np\left( \dsum \limits_{j=0}^{k-1}\frac{1}{3j+1}+\dsum%
\limits_{j=0}^{k-1}\frac{1}{3j+2}+\dsum \limits_{j=1}^{k}\frac{1}{3j}\right)
\\
&&+np\left( -\dsum \limits_{j=0}^{k-1}\frac{1}{3j+1}-2\dsum
\limits_{j=0}^{k-1}\frac{1}{3j+2}+\frac{1}{3k+1}\right) \\
&=&-1+np\left( \frac{2}{3}H_{k}+\dsum \limits_{j=0}^{k}\frac{1}{3j+1}\right) 
\text{ }\left( \func{mod}p^{2}\right) .
\end{eqnarray*}%
For the congruence (\ref{CC3}) we have%
\begin{eqnarray*}
\binom{np-1}{3k+2}_{2} &\equiv &\dsum \limits_{j=0}^{3k+2}\cos \frac{\left(
2j-2\right) \pi }{3}-2np\dsum \limits_{j=0}^{3k+2}H_{j}\cos \frac{\left(
2j-2\right) \pi }{3} \\
&=&-2np\dsum \limits_{j=0}^{3k+2}H_{j}\cos \frac{\left( 2j-2\right) \pi }{3}
\\
&=&np\left( \dsum \limits_{j=0}^{k}\left( H_{3j}-2H_{3j+1}+H_{3j+2}\right)
\right) \\
&\equiv &np\left( -\dsum \limits_{j=0}^{k}\frac{1}{3j+1}+\dsum%
\limits_{j=0}^{k}\frac{1}{3j+2}\right) \text{ }\left( \func{mod}p^{2}\right)
.
\end{eqnarray*}
\end{proof}

\section{Proof of the main results}

\begin{proof}[Proof of Theorem \protect \ref{TH}]
For $p\equiv 1$ $\left( \func{mod}3\right) $ let $3k=p-1$ in the congruence (%
\ref{CC1}). Then, by the congruences (\ref{GL}) and (\ref{C1b}) we obtain%
\begin{equation*}
\binom{np-1}{p-1}_{2}\equiv 1-np\left( \frac{2}{3}H_{\left( p-1\right)
/3}+\dsum \limits_{k=0}^{\left( p-4\right) /3}\frac{1}{3j+2}\right) \equiv
1+npq_{3}\text{ }\left( \func{mod}p^{2}\right) .
\end{equation*}%
For $p\equiv 2$ $\left( \func{mod}3\right) $ let $3k+1=p-1$ in the
congruence (\ref{CC2}). Then, by the congruences (\ref{GL}) and (\ref{C2b})
we obtain%
\begin{equation*}
\binom{np-1}{p-1}_{2}\equiv -1+np\left( \frac{2}{3}H_{\left( p-2\right)
/3}+\dsum \limits_{j=0}^{\left( p-2\right) /3}\frac{1}{3j+1}\right) \equiv
-1-npq_{3}\text{ }\left( \func{mod}p^{2}\right) .
\end{equation*}%
For $p\equiv 1$ $\left( \func{mod}6\right) $ let $3k=\left( p-1\right) /2$
in the congruence (\ref{CC1}). Then, by the congruences (\ref{GL2}) and (\ref%
{H0}) we obtain%
\begin{equation*}
\binom{np-1}{3k}_{2}\equiv 1-np\left( \frac{2}{3}H_{k}+\dsum%
\limits_{j=0}^{k-1}\frac{1}{3j+2}\right) \equiv 1+np\left( 2q_{2}+\frac{1}{2}%
q_{3}\right) \text{ }\left( \func{mod}p^{2}\right) .
\end{equation*}%
For $p\equiv 5$ $\left( \func{mod}6\right) $ let $3k+2=\left( p-1\right) /2$
in the congruence (\ref{CC3}). Then, by the congruences (\ref{H3}) and (\ref%
{H2}) we obtain%
\begin{equation*}
\binom{np-1}{\left( p-1\right) /2}_{2}\equiv np\left( -\dsum
\limits_{j=0}^{\left( p-5\right) /6}\frac{1}{3j+1}+\dsum
\limits_{j=0}^{\left( p-5\right) /6}\frac{1}{3j+2}\right) \equiv -\frac{1}{2}%
npq_{3}\text{ }\left( \func{mod}p^{2}\right) .
\end{equation*}
\end{proof}

\begin{proof}[Proof of Theorem \protect \ref{C}]
By the Known identity \cite[Eq. 2.24]{BEL} 
\begin{equation}
\binom{np-1}{p-1}_{2}=\underset{k=\left( p-1\right) /2}{\overset{p-1}{\sum }}%
\binom{np-1}{k}\binom{k}{p-1-k}  \label{R1}
\end{equation}%
we have 
\begin{eqnarray*}
\binom{np-1}{p-1}_{2} &=&\dsum \limits_{k=\left( p-1\right) /2}^{p-1}\binom{%
np-1}{k}\binom{k}{p-1-k} \\
&\equiv &\dsum \limits_{k=\left( p-1\right) /2}^{p-1}\left( -1\right)
^{k}\left( 1-npH_{k}\right) \binom{k}{p-1-k} \\
&=&\dsum \limits_{k=0}^{\left( p-1\right) /2}\left( -1\right) ^{k}\binom{%
p-1-k}{k}\left( 1-npH_{p-1-k}\right) \text{ }\left( \func{mod}p^{2}\right) .
\end{eqnarray*}%
So, by the congruence 
\begin{equation*}
\binom{p-1-k}{k}=\dbinom{p-1}{2k}\dbinom{p-1}{k}^{-1}\binom{2k}{k}\equiv
\left( -1\right) ^{k}\binom{2k}{k}\text{ }\left( \func{mod}p\right) ,\text{ }%
k\in \left \{ 0,\ldots ,\frac{p-1}{2}\right \} ,
\end{equation*}%
the identity \cite[Cor. 2.8]{DAR}%
\begin{equation}
\dsum \limits_{k=0}^{\left[ n/2\right] }\left( -1\right) ^{k}\binom{n-k}{k}%
=\left \{ 
\begin{array}{l}
\text{ \  \ }0\text{ \  \  \  \  \  \  \  \  \ if \ }n\equiv 2\text{ }\left( \func{mod%
}3\right) , \\ 
\left( -1\right) ^{\left[ n/3\right] }\text{ \  \ otherwise}%
\end{array}%
\right.  \label{R2}
\end{equation}%
and the congruence (\ref{CONG 0}) we get%
\begin{equation*}
\binom{np-1}{p-1}_{2}\equiv \left( -1\right) ^{\left[ \left( p-1\right) /3%
\right] }-np\dsum \limits_{k=0}^{\left( p-1\right) /2}\binom{2k}{k}H_{k}%
\text{ }\left( \func{mod}p^{2}\right) .
\end{equation*}%
We note here that $\left( -1\right) ^{\left[ \left( p-1\right) /3\right] }=1$
if $p\equiv 1$ $\left( \func{mod}3\right) $ and $-1$ if $p\equiv 2$ $\left( 
\func{mod}3\right) .$ \newline
Hence, by combining this congruence with the congruence (\ref{2}), we obtain
the congruence (\ref{6}). Set $L=np-1,m=\left( p-1\right) /2$ in (\ref{R1}).
Then we have%
\begin{eqnarray*}
\binom{np-1}{\left( p-1\right) /2}_{2} &=&\dsum \limits_{k=0}^{\left(
p-1\right) /2}\binom{np-1}{k}\binom{k}{\left( p-1\right) /2-k} \\
&\equiv &\dsum \limits_{k=0}^{\left( p-1\right) /2}\left( -1\right)
^{k}\left( 1-npH_{k}\right) \binom{k}{\left( p-1\right) /2-k}
\end{eqnarray*}%
\begin{eqnarray*}
&=&\left( -1\right) ^{\left( p-1\right) /2}\dsum \limits_{k=0}^{\left(
p-1\right) /2}\left( -1\right) ^{k}\binom{\left( p-1\right) /2-k}{k} \\
&&-\left( -1\right) ^{\left( p-1\right) /2}np\dsum \limits_{k=0}^{\left(
p-1\right) /2}\left( -1\right) ^{k}\binom{\left( p-1\right) /2-k}{k}%
H_{\left( p-1\right) /2-k}\text{ }\left( \func{mod}p^{2}\right) .
\end{eqnarray*}%
Then, by the congruence (\ref{CONG 1}) the last congruence can be written as%
\begin{eqnarray*}
\binom{np-1}{\left( p-1\right) /2}_{2} &\equiv &\left( -1\right) ^{\frac{p-1%
}{2}}\left( 1+2npq_{2}\right) \dsum \limits_{k=0}^{\left[ \frac{p-1}{4}%
\right] }\left( -1\right) ^{k}\binom{\frac{p-1}{2}-k}{k} \\
&&-\left( -1\right) ^{\frac{p-1}{2}}np\dsum \limits_{k=1}^{\left[ \frac{p-1}{%
4}\right] }\left( -1\right) ^{k}\binom{\frac{p-1}{2}-k}{k}\left(
2H_{2k}-H_{k}\right) \text{ }\left( \func{mod}p^{2}\right) .
\end{eqnarray*}%
But for $k\in \left \{ 1,2,\ldots ,\left[ \left( p-1\right) /4\right]
\right
\} $ we have%
\begin{eqnarray*}
\left( -1\right) ^{k}\binom{\frac{p-1}{2}-k}{k} &=&\left( -1\right) ^{k}%
\frac{\left( p-\left( 2k+1\right) \right) \left( p-\left( 2k+3\right)
\right) \cdots \left( p-\left( 4k-1\right) \right) }{2^{k}k!} \\
&\equiv &\frac{\left( 2k+1\right) \left( 2k+3\right) \cdots \left(
4k-1\right) }{2^{k}k!} \\
&=&\frac{1}{4^{k}}\dbinom{4k}{2k}\text{ }\left( \func{mod}p\right) ,
\end{eqnarray*}%
hence%
\begin{eqnarray}
\binom{np-1}{\left( p-1\right) /2}_{2} &\equiv &\left( -1\right) ^{\frac{p-1%
}{2}}\left( 1+2npq_{2}\right) \dsum \limits_{k=0}^{\left[ \frac{p-1}{4}%
\right] }\left( -1\right) ^{k}\binom{\frac{p-1}{2}-k}{k}  \label{W} \\
&&-\left( -1\right) ^{\frac{p-1}{2}}np\dsum \limits_{k=1}^{\left[ \frac{p-1}{%
4}\right] }\frac{1}{4^{k}}\dbinom{4k}{2k}\left( 2H_{2k}-H_{k}\right) \text{ }%
\left( \func{mod}p^{2}\right) .  \notag
\end{eqnarray}%
Then, for $p\equiv 1$ $\left( \func{mod}6\right) ,$ the identity (\ref{R2})
shows that we have%
\begin{equation*}
\dsum \limits_{k=0}^{\left( p-1\right) /2}\left( -1\right) ^{k}\binom{\left(
p-1\right) /2-k}{k}=\left( -1\right) ^{\left[ \left( p-1\right) /6\right] },
\end{equation*}%
and since $\left( -1\right) ^{\left( p-1\right) /2+\left[ \left( p-1\right)
/6\right] }=1,$ the congruence (\ref{W}) becomes%
\begin{equation*}
\binom{np-1}{\left( p-1\right) /2}_{2}\equiv 1+2npq_{2}-\left( -1\right)
^{\left( p-1\right) /2}np\dsum \limits_{k=1}^{\left[ \left( p-1\right) /4%
\right] }\frac{1}{4^{k}}\dbinom{4k}{2k}\left( 2H_{2k}-H_{k}\right) \text{ }%
\left( \func{mod}p^{2}\right) ,
\end{equation*}%
and by combining this congruence with the congruence (\ref{4}), we obtain
the congruence (\ref{7}). Also, for $p\equiv 5$ $\left( \func{mod}6\right) ,$
the identity (\ref{R2}) shows that we have%
\begin{equation*}
\dsum \limits_{k=0}^{\left( p-1\right) /2}\left( -1\right) ^{k}\binom{\left(
p-1\right) /2-k}{k}=0,
\end{equation*}%
so the congruence (\ref{W}) becomes%
\begin{equation*}
\binom{np-1}{\left( p-1\right) /2}_{2}\equiv -\left( -1\right) ^{\left(
p-1\right) /2}np\dsum \limits_{k=1}^{\left[ \left( p-1\right) /4\right] }%
\frac{1}{4^{k}}\dbinom{4k}{2k}\left( 2H_{2k}-H_{k}\right) \text{ }\left( 
\func{mod}p^{2}\right) ,
\end{equation*}%
and by combining this congruence with the congruence (\ref{4}), we obtain
the congruence (\ref{7}).
\end{proof}

\begin{proof}[Proof of Proposition \protect \ref{PP}]
For $k\in \left \{ 0,1,\ldots ,\left[ p/3\right] -1\right \} ,$ from
Proposition \ref{P} we may state%
\begin{equation}
\binom{np-1}{3k}_{2}+\binom{np-1}{3k+1}_{2}+\binom{np-1}{3k+2}_{2}\equiv 
\frac{np}{3k+2}\text{ }\left( \func{mod}p^{2}\right) .  \label{a}
\end{equation}%
To prove the congruences (\ref{9}) let%
\begin{equation*}
\overset{p-1}{\underset{j=0}{\sum }}\binom{np-1}{j}_{2}=\overset{\left[
\left( p-1\right) /3\right] }{\underset{j=0}{\sum }}\binom{np-1}{3j}_{2}+%
\overset{\left[ \left( p-2\right) /3\right] }{\underset{j=0}{\sum }}\binom{%
np-1}{3j+1}_{2}+\overset{\left[ \left( p-3\right) /3\right] }{\underset{j=0}{%
\sum }}\binom{np-1}{3j+2}_{2}.
\end{equation*}%
For $p\equiv 1$ $\left( \func{mod}3\right) ,$ by the congruences (\ref{a}), (%
\ref{2}) and (\ref{C1b}), we get%
\begin{eqnarray*}
\overset{p-1}{\underset{j=0}{\sum }}\binom{np-1}{j}_{2} &=&\overset{\left(
p-1\right) /3}{\underset{j=0}{\sum }}\binom{np-1}{3j}_{2}+\overset{\left(
p-1\right) /3-1}{\underset{j=0}{\sum }}\binom{np-1}{3j+1}_{2}+\overset{%
\left( p-1\right) /3-1}{\underset{j=0}{\sum }}\binom{np-1}{3j+2}_{2} \\
&\equiv &\binom{np-1}{p-1}_{2}+\overset{\left( p-4\right) /3}{\underset{j=0}{%
\sum }}\left[ \binom{np-1}{3j}_{2}+\binom{np-1}{3j+1}_{2}\binom{np-1}{3j+2}%
_{2}\right] \\
&\equiv &\binom{np-1}{p-1}_{2}+np\overset{\left( p-4\right) /3}{\underset{j=0%
}{\sum }}\frac{1}{3j+2} \\
&\equiv &1+npq_{3}\text{ }\left( \func{mod}p^{2}\right) .
\end{eqnarray*}%
For $p\equiv 2$ $\left( \func{mod}3\right) ,$ by the congruences (\ref{a}), (%
\ref{CC1}), (\ref{GL}), (\ref{2}) and (\ref{C2c}), we get%
\begin{eqnarray*}
\overset{p-1}{\underset{j=0}{\sum }}\binom{np-1}{j}_{2} &=&\overset{\left(
p-2\right) /3}{\underset{j=0}{\sum }}\binom{np-1}{3j}_{2}+\overset{\left(
p-2\right) /3}{\underset{j=0}{\sum }}\binom{np-1}{3j+1}_{2}+\overset{\left(
p-2\right) /3-1}{\underset{j=0}{\sum }}\binom{np-1}{3j+2}_{2} \\
&\equiv &\binom{np-1}{p-2}_{2}+\binom{np-1}{p-1}_{2}+np\overset{\left(
p-5\right) /3}{\underset{j=0}{\sum }}\frac{1}{3j+2} \\
&\equiv &1-np\left( \frac{2}{3}H_{\left( p-2\right) /3}+\dsum
\limits_{j=0}^{\left( p-5\right) /3}\frac{1}{3j+2}\right) +\left(
-1-npq_{3}\right) +\left( \frac{1}{2}npq_{3}\right) \\
&\equiv &1-np\left( \frac{2}{3}\left( -\frac{3}{2}q_{3}\right) +\frac{1}{2}%
q_{3}\right) -1-\frac{1}{2}npq_{3} \\
&=&0\text{ }\left( \func{mod}p^{2}\right) .
\end{eqnarray*}%
To prove the congruences (\ref{10}) let%
\begin{equation*}
\overset{\left( p-1\right) /2}{\underset{j=0}{\sum }}\binom{np-1}{j}_{2}=%
\overset{\left[ \left( p-1\right) /6\right] }{\underset{j=0}{\sum }}\binom{%
np-1}{3j}_{2}+\overset{\left[ \left( p-3\right) /6\right] }{\underset{j=0}{%
\sum }}\binom{np-1}{3j+1}_{2}+\overset{\left[ \left( p-5\right) /6\right] }{%
\underset{j=0}{\sum }}\binom{np-1}{3j+2}_{2}.
\end{equation*}%
For $p\equiv 1$ $\left( \func{mod}6\right) ,$ by the congruences (\ref{a}), (%
\ref{2}) and (\ref{H1}), we obtain%
\begin{eqnarray*}
\overset{\left( p-1\right) /2}{\underset{k=0}{\sum }}\binom{np-1}{k}_{2} &=&%
\overset{\left( p-1\right) /6}{\underset{j=0}{\sum }}\binom{np-1}{3j}_{2}+%
\overset{\left( p-1\right) /6-1}{\underset{j=0}{\sum }}\binom{np-1}{3j+1}%
_{2}+\overset{\left( p-1\right) /6-1}{\underset{j=0}{\sum }}\binom{np-1}{3j+2%
}_{2} \\
&\equiv &\binom{np-1}{\left( p-1\right) /2}_{2}+\overset{\left( p-7\right) /6%
}{\underset{j=0}{\sum }}\left[ \binom{np-1}{3j}_{2}+\binom{np-1}{3j+1}_{2}+%
\binom{np-1}{3j+2}_{2}\right] \\
&\equiv &\binom{np-1}{\left( p-1\right) /2}_{2}+np\overset{\left( p-7\right)
/6}{\underset{j=0}{\sum }}\frac{1}{3j+2} \\
&\equiv &1+np\left( 2q_{2}+\frac{1}{2}q_{3}\right) +np\left( -\frac{2}{3}%
q_{2}+\frac{1}{2}q_{3}\right) \\
&=&1+np\left( \frac{4}{3}q_{2}+q_{3}\right) \text{ }\left( \func{mod}%
p^{2}\right) .
\end{eqnarray*}%
For $p\equiv 5$ $\left( \func{mod}6\right) ,$ by the congruences (\ref{a})
and (\ref{H1}), we obtain%
\begin{eqnarray*}
\overset{\left( p-1\right) /2}{\underset{k=0}{\sum }}\binom{np-1}{k}_{2} &=&%
\overset{\left( p-5\right) /6}{\underset{j=0}{\sum }}\binom{np-1}{3j}_{2}+%
\overset{\left( p-5\right) /6}{\underset{j=0}{\sum }}\binom{np-1}{3j+1}_{2}+%
\overset{\left( p-5\right) /6}{\underset{j=0}{\sum }}\binom{np-1}{3j+2}_{2}
\\
&\equiv &\overset{\left( p-5\right) /6}{\underset{j=0}{\sum }}\left[ \binom{%
np-1}{3j}_{2}+\binom{np-1}{3j+1}_{2}+\binom{np-1}{3j+2}_{2}\right] \\
&\equiv &np\overset{\left( p-5\right) /6}{\underset{j=0}{\sum }}\frac{1}{3j+2%
} \\
&\equiv &-\frac{2}{3}npq_{2}\text{ }\left( \func{mod}p^{2}\right) .
\end{eqnarray*}
\end{proof}

\begin{proof}[Proof of Corollary \protect \ref{CC}]
Corollary \ref{CC} follows easily from Proposition \ref{P}.
\end{proof}

\end{document}